\documentclass[final,3p,times]{elsarticle}
\usepackage{amsmath,latexsym,amssymb,amsfonts,amsbsy}
\usepackage{tikz}
\usetikzlibrary{arrows,positioning}
\usetikzlibrary{quotes,angles}
\usetikzlibrary{calc}
\usetikzlibrary{decorations.pathreplacing,calligraphy}
\usepackage{epstopdf}
\usepackage{amsthm, bm}
\usepackage{cases}
 \usepackage{mathrsfs}
 \usepackage{graphicx}
\usepackage{graphicx,color}
\usepackage{graphicx,graphics}
\DeclareGraphicsExtensions{.eps,.bmp}
\usepackage{CJK}
\usepackage{subfigure}
\usepackage{indentfirst}%
\usepackage{verbatim}%commet
\usepackage{multirow}
\usepackage{booktabs}
\usepackage[ruled,vlined]{algorithm2e}
\usepackage{float}
\allowdisplaybreaks[4]

\usepackage{lineno}
\newtheorem{theorem}{Theorem}

\newtheorem{remark}{Remark}

\newtheorem{lemma}{Lemma}

\newtheorem{assumption}{Assumption}

\biboptions{numbers,sort&compress}
\begin{document}
\begin{frontmatter}
%% A short note
\title{A high-order rectilinear Lagrangian method  based on the geometric conservation law
}
%%\author{Xiao Li, Jiayin Zhai, Zhijun Shen\corref{cor1}}
\author{Xun Wang$^{1}$}
%% First author's email
\ead{s151025@muc.edu.cn}

\author{Chengdi Ma\corref{cor1}$^{2}$}
\ead{mcd123@mail.ustc.edu.cn}

 \cortext[cor1]{
Corresponding author at: School of Mathematical Sciences, Peking University, Beijing, 100871, PR China}

% \author{Hongping Guo$^{2,4}$}
% \ead{gguohongping@163.com}
% \author{Zhijun Shen \corref{cor1}$^{2,3}$}
% \cortext[cor1]{Corresponding author:
%   Tel.: 86-10-61935178.}
% \ead{shen\_zhijun@iapcm.ac.cn}
%%\author{Liqi Liu$^{2}$}
%\cortext[cor1]{Corresponding author:
%  Tel.: 86-10-61935178.}
%% Third author's email
%\ead{dai_zihuan@iapcm.ac.cn}
%\tnotetext[tnote1]{This project was supported by the National Natural Science Foundation of China (11971071).}
%\author{Xun Wang$^{1}$}
%\ead{wangxun18@gscaep.ac.cn}
%\author{Zhijun Shen \corref{cor1} $^{1}$}
%\ead{shen\_zhijun@iapcm.ac.cn} \cortext[cor1]{Corresponding
%author. Tel.: 86-10-61935178}
\address{ 
1. Academy of Mathematics and Systems Science,
Chinese Academy of Sciences,  Beijing,
100190, PR China\\
2. School of Mathematical Sciences, Peking University, Beijing 100871,PR  China
%   2. National Key Laboratory of Computational Physics, Institute of Applied Physics and Computational Mathematics,
%  P. O. Box 8009-26, Beijing 100088, China\\
%  3. HEDPS, Center for Applied Physics and Technology, and College of
% Engineering, Peking University, Beijing 100871, China\\
% 4. Faculty of Mathematics, Baotou Teachers' College, Bao tou 014030, China
}
%\title{}
%\author{}
%\address[1]{Graduate School of China Academy of Engineering Physics, P. O. Box 8009, Beijing 100088, P. R. China}
%\address[2]{Institute of Applied Physics and Computational Mathematics, P.O. Box 8009, Beijing 100088, PR China}
\begin{abstract}
This paper presents a mesh moving strategy for high-order Lagrangian method on quadrilateral meshes. The primary evidence of this method stems from principle of area conservative linearization and the asymptotic properties of the velocity. The former strictly adheres to the requirements of geometric conservation laws, while the latter provides a high-order accuracy guarantee.
%The former strictly adheres to the requirements of geometric conservation laws, while the latter  offers a highly accurate theoretical guarantee. 
Two smooth vortex test cases verify the feasibility of the proposed scheme. %Two smooth vortex cases enable us to verify our scheme feasibility. 
\end{abstract}

\end{frontmatter}

% \linenumbers
% \modulolinenumbers[5]
% \pagewiselinenumbers
% \switchlinenumbers

\section{Introduction}
%移动网格
%拉氏方法扮演..., 也是ALE方法的....高阶拉氏...,但是 曲边网格的痛点， 能否够构造直边网格移动
The mesh moving method is an indispensable tool for studying the macroscopic dynamics of fluids with moving boundaries or deforming interfaces~\cite{DuanJ, LiS}. Especially for the Lagrangian methods~\cite{Maire}, shock-capturing schemes are powerful means to track material interfaces and also the cornerstone of Arbitrary Lagrangian-Eulerian methods~\cite{BoscheriM, BoscheriM1}. High-order Lagrangian methods provide more detailed representations of physical fields and have been a research focus in recent years~\cite{DobrevV, FuP}. 

However, high-order schemes often require curvilinear meshes to represent the instantaneous shape of the fluid~\cite{LiuX, MorganN},  which severely complicates the satisfaction of geometric conservation law, degrades mesh quality, and increases computational complexity. It is natural to consider constructing a high-order Lagrangian method on rectilinear moving meshes, which would greatly improve both the compatibility of the algorithm with existing first- or second-order schemes and the simplicity of the algorithm itself. Thus, we propose a third-order rule that determines the positions of values on the edges connected to the nodes and the associated weights. We then apply the improved CAVEAT-type method to obtain the mesh moving velocity.

\section{Geometric conservation law}
In the Lagrangian framework, the variation in time of a control volume needs to satisfy the geometric conservation law (GCL)
\begin{eqnarray}\label{GQL}
\frac{\mathrm{d}}{\mathrm{d} t}\int_{\Omega}\mathrm{d} \Omega- \int_{\partial \Omega} {\bf u}\cdot {\bf n}\mathrm{d} l= 0, 
\end{eqnarray} 
where $\Omega$ is the control volume, ${\bf u}$ is the moving velocity and ${\bf n}$  is the outward normal vector of the boundary $\partial \Omega$ respectively. 

Suppose the initial computational domain $\Omega$ is discretized into $N_c$  quadrilateral  cells $\{\Omega_{c},c=1,\cdots,N_c\}$, and each node in the cell is labeled by an index $q$, whose coordinate is ${\bf x}_{q}$, shown in \ref{fig:grid1}. To describe the relationships between cells and nodes, we define the following sets:
\begin{eqnarray*}
\begin{aligned}
\mathcal{Q}(c) &:= \{ \text{counterclockwise-ordered list of nodes in the cell } \Omega_c \}, \\
\mathcal{N}(q) &:= \{ \text{nodes adjacent to the node } q \}.
% \\
% \mathcal{C}(q) &:= \{ \text{cells that share the node } q \}.
\end{aligned}
\end{eqnarray*}

\begin{figure}[!htp] 
\centering
\subfigure[]{
\label{fig:grid1}
\begin{tikzpicture}[scale=1.2],
\coordinate (q1) at (5.5,3.8);
\coordinate (q11) at (4.5,4.0);
\coordinate (q12) at (5.5,2.9);
\coordinate (q2) at (8.2,5);
\coordinate (q21) at (8.1,4.1);
\coordinate (q22) at (8.8,5.3);
\coordinate (q3) at (7.5,7);
\coordinate (q31) at (7.0,7.5);
\coordinate (q32) at (8.2,7.6);
\coordinate (q4) at (5,6);
\coordinate (q41) at (5.5,6.7);
\coordinate (q42) at (4.2,6.2);
\coordinate (qc) at (6.5,5.5);
\coordinate (e1s1) at (6.6,4.3);
\coordinate (e1s2) at (6.8,3.8);
\coordinate (e2s1) at (8.0,5.5);
\coordinate (e2s2) at (8.5,5.7);
\coordinate (e2s3) at (7.8,6.2);
\coordinate (e2s4) at (7.2,5.9);
\coordinate (O) at (7,2.5);
\draw [color = black, line width =1.4pt]
(q1) -- (q2) -- (q3)-- (q4)--(q1) 
(q12)--(q1) -- (q11) 
(q22)--(q2) -- (q21) 
(q31)--(q3) -- (q32)
(q41)--(q4) -- (q42);
\draw[->,red,line width =1.5pt] (e1s1) -- (e1s2);
\draw[->,red,line width =1.5pt] (e2s1) -- (e2s2);
% \draw[->, brown,line width =1.5pt] (e2s3) -- (e2s4);
\draw[densely dashed,blue,->,line width =1.4pt]
(6.0,5.0) arc (120: 390: 0.3);
\filldraw [black] (q1)circle (2.5pt);
\filldraw [black] (q2)circle (2.5pt);
\filldraw [black] (q3)circle (2.5pt);
\filldraw [black] (q4)circle (2.5pt);
%====Neribourhood cells
\filldraw [black] (q11)circle (2.5pt);
\filldraw [black] (q12)circle (2.5pt);
\filldraw [black] (q21)circle (2.5pt);
\filldraw [black] (q22)circle (2.5pt);
\filldraw [black] (q31)circle (2.5pt);
\filldraw [black] (q32)circle (2.5pt);
\filldraw [black] (q41)circle (2.5pt);
\filldraw [black] (q42)circle (2.5pt);
% \filldraw [black] (q2s1)circle (2.5pt);
%==== edge q^-q=e1
% \filldraw [black] (qc)circle (2.5pt);
\node [below of=qc,xshift=0.0cm,yshift=0.8cm,black]{$\Omega_c$};
% \node [below of=e2s2,xshift=0.5cm,yshift=1.8cm,black]{$\Omega_f$};
\node [below of=q1,xshift=-0.3cm,yshift=0.8cm,black]{$q^-$};
\node [below of=q2,xshift=0.3cm,yshift=0.8cm,black]{$q$};
\node [below of=q3,xshift= 0.1cm,yshift=1.4cm,black]{$q^+$};
\node [below of=q4,xshift=0.0cm,yshift=1.4cm,black]{$q^{++}$};
 % \node [below of=O,yshift=0.70cm,black]{(a) };
%===normal vector
\node [below of=e2s2,xshift=0.8cm,yshift=1.0cm,black]{$
{ \color{red} \mathbf{n}_{qq^+} } $};
\node [below of=e1s2,xshift=0.5cm,yshift=1.0cm,black]{${\color{red} \mathbf{n}_{q^-q}} $};
% \node [below of=e2s4,xshift=0.3cm,yshift=0.8cm,black]{${\color{brown} \bs{n}_c^f} $};
\end{tikzpicture}\qquad
}
\subfigure[]{
\label{fig:grid2}
\begin{tikzpicture}[scale=0.9], 
\coordinate (q1) at (5.5,3.8);
\coordinate (q2) at (8.2,5.2);
\coordinate (q3) at (7.5,7.5);
\coordinate (q4) at (5,6);
\coordinate (q5) at (6.5,2);
\coordinate (q6) at (9,2.5);
\coordinate (q7) at (11.5,3);
\coordinate (q8) at (11.5,6);
\coordinate (q9) at (10.5,8);
\coordinate (qc) at (6.5,5.5);
\coordinate (e1s1) at (6.8,4.5);
\coordinate (e1s2) at (6.6,4.9);
\coordinate (e2s1) at (8.0,6);
\coordinate (e2s2) at (8.5,6.2);
\coordinate (e3s1) at (9.8,5.6);
\coordinate (e3s2) at (9.9,5.1);
\coordinate (e4s1) at (8.5,4.2);
\coordinate (e4s2) at (7.9,4.0);
\coordinate (e5s1) at (6.6,5.8);
\coordinate (e5s2) at (6.2,6.1);
\coordinate (e6s1) at (10,6.5);
\coordinate (e6s2) at (10.2,7.0);
\coordinate (e7s1) at (10.0,4.0);
\coordinate (e7s2) at (10.5,3.6);
\coordinate (e8s1) at (7.1,3.2);
\coordinate (e8s2) at (6.9,2.6);
\coordinate (O) at (9,1);
\draw [color = black, line width =1.4pt]
(q1) -- (q2) -- (q3)-- (q4)--(q1) (q1) --(q5)--(q6)--(q7)--(q8)--(q9)--(q3) (q6)--(q2) -- (q8);
% \draw [color = red, dashed,line width =1.5pt]
% (q1) --(q3)--(q8)--(q6)--(q1);
\draw[->,red,line width =1.5pt] (e1s1) -- (e1s2);
\draw[->,red,line width =1.5pt] (e2s1) -- (e2s2);
\draw[->,red,line width =1.5pt] (e3s1) -- (e3s2);
\draw[->,red,line width =1.5pt] (e4s1) -- (e4s2);
% \draw[->,brown,line width =1.5pt] (e5s1) -- (e5s2);
% \draw[->,brown,line width =1.5pt] (e6s1) -- (e6s2);
% \draw[->,brown,line width =1.5pt] (e7s1) -- (e7s2);
% \draw[->,brown,line width =1.5pt] (e8s1) -- (e8s2);
\draw[densely dashed,blue,->,line width =1.4pt]
(8.8,5.1) arc (360: 80: 0.6);
\filldraw [black] (q1)circle (2.5pt);
\filldraw [black] (q2)circle (2.5pt);
\filldraw [black] (q3)circle (2.5pt);
\filldraw [black] (q4)circle (2.5pt);
\filldraw [black] (q5)circle (2.5pt);
\filldraw [black] (q6)circle (2.5pt);
\filldraw [black] (q7)circle (2.5pt);
\filldraw [black] (q8)circle (2.5pt);
\filldraw [black] (q9)circle (2.5pt);
%====Neribourhood cells
% \filldraw [black] (q11)circle (2.5pt);
% \filldraw [black] (q12)circle (2.5pt);
% \filldraw [black] (q21)circle (2.5pt);
% \filldraw [black] (q22)circle (2.5pt);
% \filldraw [black] (q31)circle (2.5pt);
% \filldraw [black] (q32)circle (2.5pt);
% \filldraw [black] (q41)circle (2.5pt);
% \filldraw [black] (q42)circle (2.5pt);
% \filldraw [black] (q2s1)circle (2.5pt);
%==== edge q^-q=e1
% \filldraw [black] (qc)circle (2.5pt);
\node [below of=qc,xshift=-0.5cm,yshift=0.8cm,black]{$\Omega_c$};
\node [below of=e2s2,xshift=0.5cm,yshift=2.0cm,black]{$\Omega_{c^+}$};
\node [below of=qc,xshift=1.2cm,yshift=-1.7cm,black]{$\Omega_{c^-}$};
\node [below of=e2s2,xshift=1.5cm,yshift=-1.9cm,black]{$\Omega_{c^{++}}$};
\node [below of=q1,xshift=-0.4cm,yshift=0.7cm,black]{$q^4$};
\node [below of=q2,xshift=0.3cm,yshift=0.8cm,black]{$q$};
\node [below of=q3,xshift= 0.1cm,yshift=1.4cm,black]{$q^1$};
 \node [below of=q8,xshift=0.2cm,yshift=1.4cm,black]{$q^{2}$};
\node [below of=q6,xshift=0.2cm,yshift=0.8cm,black]{$q^{3}$};
 % \node [below of=O,yshift=0.70cm,black]{(b) };
%===normal vector
\node [below of=e2s2,xshift=0.3cm,yshift=1.1cm,black]{$
{ \color{red} \mathbf{n}_1 } $};
\node [below of=e1s2,xshift=0.4cm,yshift=1.1cm,black]{${\color{red} \mathbf{n}_{4}} $};
\node [below of=e3s2,xshift=0.4cm,yshift=1.1cm,black]{${\color{red} \mathbf{n}_{2}} $};
\node [below of=e4s2,xshift=0.3cm,yshift=0.75cm,black]{${\color{red} \mathbf{n}_3} $};
% \node [below of=e5s2,xshift=0.3cm,yshift=1.25cm,black]{${\color{brown} \bs{n}_{qc}} $};
% \node [below of=e6s2,xshift=0.38cm,yshift=1.1cm,black]{${\color{brown} \bs{n}_{qc^+}} $};
% \node [below of=e7s2,xshift=0.5cm,yshift=0.9cm,black]{${\color{brown} \bs{n}_{qc^{++}}} $};
% \node [below of=e8s2,xshift=0.1cm,yshift=0.75cm,black]{${\color{brown} \bs{n}_{qc^-}} $};
\end{tikzpicture}
}
\caption{Geometrical notation. (a) Primary cell $\Omega_c$; (b) The normal vectors of the edges connected to node $q$ are aligned in clockwise order.} \label{fig:grid}
\end{figure}
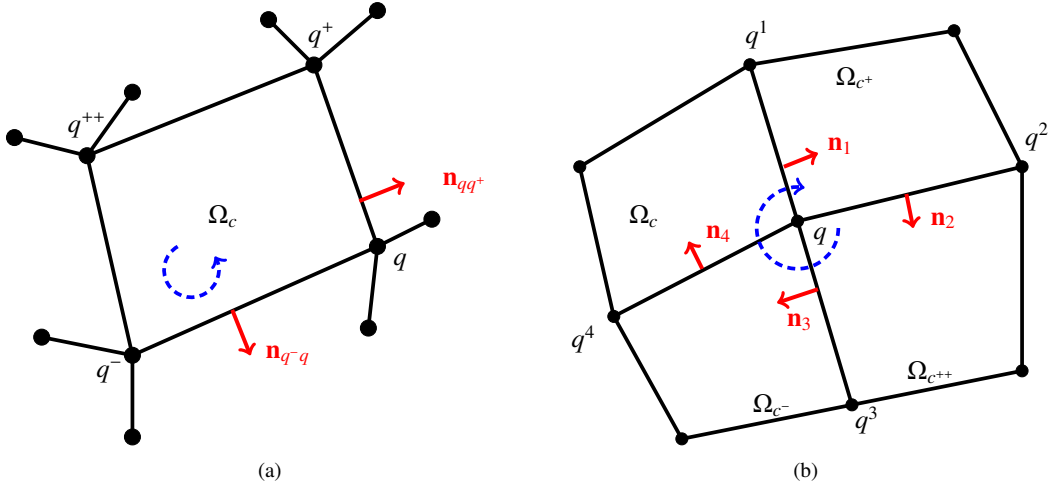

Based on Lagrangian assumption, the mass of the cell $\Omega_c$, defined as $m_c: = \int _{\Omega_c} \rho \mathrm{d}\Omega$ remains unchanged over time.  
Then, let us discretize Eqs.(\ref{GQL}) using a finite-volume numerical scheme on the cell $\Omega_{c}$:
 \begin{eqnarray} \label{GQL_dis}
|\Omega_{c}^{n+1}| = |\Omega_{c}^{n}| + \frac{\Delta t}{2}  \sum_ {q \in \mathcal{Q}(c)}  \left(l_{qq^+} {\bf n}_{qq^+} +l_{q^-q} {\bf n}_{q^-q}  \right ) \cdot {\bf u}_q^*,
\end{eqnarray}
where  $\Delta t := t^{n+1}-t^n$ is the time interval, $l_k,(k=q^-q,qq^+)$ is the length of the cell edge $k$ at time $t^n$. Here  ${\bf u}_{q}^{*}$ represents the velocity on the vertex $q$. The readers may refer to Ref.~\cite{WangX} 

% By now, we consider a smooth physical field. 
The cell density $\rho_{c}$ and specific volume are
\begin{equation}
\rho_{c}^{n+1} = \frac {m_{c}}{|\Omega_{c}^{n+1}|}, \quad \tau_{c}^{n+1} =  \frac {|\Omega_{c}^{n+1}|} {m_{c}} = \frac {1}{\rho_{c}^{n+1}}. 
\end{equation}

\subsection{Mesh moving method}
In hydrodynamics, the CAVEAT algorithm offers a mesh evolution strategy that alters the mesh in a Lagrangian style~\cite{WangX}. In general, shown in \ref{fig:grid2}, the nodal velocity ${\bf u}_q^*$ can be obtained by minimizing the following quadratic functional for each vertex $q$,
\begin{eqnarray} \label{CAVEAT}
%Func({\bf u}_q^* )
{F({\bf u}_q^*)}=\sum_{q' \in \mathcal{N}(q)} l_{k}
({{\bf u}_q^*\cdot } {\bf n}_{k}-S_{k}^*)^2,k=qq',
\end{eqnarray}
where $S_k^*$ is the contact discontinuity velocity of boundary $l_k$. We solve the Eq. (\ref{CAVEAT}) to get the vertex velocity ${{\bf u}_q^*}$,

\begin{eqnarray} \label{CAVEAT_node}
%Func({\bf u}_q^* )
{{\bf u}_q^*}= \mathcal{M}_q\sum_{q' \in \mathcal{N}(q)} l_{k}
\alpha_kS_k^*{\bf n}_{k} ,k=qq',
\end{eqnarray}
where $\alpha_k$ is the acoustic impedance, $\mathcal{M}_q$ is nodal  matrix
\begin{equation}
 \mathcal{M}_q=\left[
\begin{array}{cc}
\sum_{q' \in \mathcal{N}(q)} l_{k}
\alpha_k n_{x,k}^2 &  \sum_{q' \in \mathcal{N}(q)} l_{k}
\alpha_k n_{x,k}n_{y,k}\\
\sum_{q' \in \mathcal{N}(q)} l_{k}
\alpha_k n_{y,k}n_{x,k} & \sum_{q' \in \mathcal{N}(q)} l_{k}
\alpha_k n_{y,k}^2
\end{array}
\right].
\end{equation}
Here we set $\alpha_k=1$ without considering shock situation.
The node position at the next time is obtained
 \begin{eqnarray}\label{mesh_move}
  {\bf x}_q^{n+1}={\bf x}_q^{n} + \Delta t {\bf u}_{q}^*.
 \end{eqnarray}
Each cell volume  of  $|\Omega_{c}^{n+1}|$ is evaluated directly from the updated vertex coordinates.

\subsection{CFL-condition}
Time step  limitations: Maire et al.'s time algorithm\cite{}.
Define the next time step $\Delta t^{n+1} =t^{n+1}-t^{n}$.  It is given  $\Delta t^{n+1} = \min\{\Delta t_{E},\Delta t_{V}, \Delta t_{M}\}$ with
\begin{equation} \label{CFL}
\Delta t_{E}=C_{E} \min _{c} \frac{\lambda_c}{a_{max}}, \quad \Delta t_{V}=C_{V} \min _{c}\left(
\frac{\left|\Omega_{c}^n\right|}{\left|\frac{d}{d t}\Omega_{c}^n \right |}\right),\quad  \Delta t_{M}=C_M \Delta t^n.
\end{equation}
% where  $C_{E}=0.5$, $C_{v}=0.1$, $C_{M}=1.01$ are constants. 
Here $\Delta t^n$ is the current time step. At the initial time  $\Delta t^0 = 10^{-8}$. 
$C_{E}=0.5$ is  a CFL-like criterion  in order to ensure a
positive entropy production in cell $\Omega_c$ during the time step.
$\lambda_{c} = \sup_{\bm x,\,\bm y \in \Omega_c^n} \|\bm x - \bm y\|$ is the diameter of $\Omega_c^n$, and $a_{max}^n$ is the maximum acoustic speeds of fast wave across boundaries, $C_{V}=0.1$ is the criterion on the variation of volume  which limits the relative change in the cell volume to be less than $0.1$.
 $C_{M}=1.01$ is the coefficient  that allows the time step to increase.

\section{High-order mesh moving strategy}\label{sece3}

\subsection{Principle of area-conservative linearization} \label{area-correction}
\begin{lemma}\label{lemma_1}
Let $v(x)$ be a continuous velocity function defined on the interval $[a, b]$, with a known definite integral (accumulated quantity or area) denoted as $S_{target}$. If a linear function $L(x) = kx + c$ is constructed to approximate $v(x)$ subject to the following two geometric constraints:
\begin{enumerate}
    \item \textbf{Slope Preservation Constraint:} The slope of the linear function $L(x)$ is identical to the slope of the secant line connecting the endpoints of the original function, i.e., $k = \frac{v(b) - v(a)}{b - a}$.
    \item \textbf{Area Conservation Constraint:} The definite integral of the linear function $L(x)$ over the interval $[a, b]$ is equal to $S_{target}$.
\end{enumerate}
Then, the endpoint values of the linear function, denoted as $v_a' = L(a)$ and $v_b' = L(b)$, are uniquely determined by:
\begin{equation}
    v_a' = v(a) + \Delta S,\quad 
    v_b' = v(b) + \Delta S,
\end{equation}
where $\Delta S$ is the area correction term, given by:
\begin{equation}
    \Delta S = \frac{S_{target}}{b - a} - \frac{v(a) + v(b)}{2}.
\end{equation}
\end{lemma}

\begin{proof}
Let the endpoint values of $L(x)$ be $v_a'$ and $v_b'$. 
From the Slope Preservation Constraint, the slope of the new line $k_{new}$ must equal the slope of the original secant $k_{old}$:
\begin{equation}
    \frac{v_b' - v_a'}{b - a} = \frac{v(b) - v(a)}{b - a} \implies v_b' - v_a' = v(b) - v(a).
\end{equation}
This implies that the difference between endpoint values remains constant, meaning $L(x)$ is obtained by a vertical translation of the original secant line $L_{old}(x)$. Let the translation magnitude be $\Delta S$, such that:
\begin{equation}
    v_a' = v(a) + \Delta S, \quad v_b' = v(b) + \Delta S.
\end{equation}
From the Area Conservation Constraint, the trapezoidal area must equal $S_{target}$:
\begin{equation}
    \int_{a}^{b} L(x) \, dx = \frac{v_a' + v_b'}{2} (b - a) := S_{target}.
\end{equation}
Solving for $\Delta S$:
\begin{equation}
    \Delta S = \frac{S_{target}}{b - a} - \frac{v(a) + v(b)}{2},
\end{equation}
which completes the proof.
\end{proof}
% \Delta S 是边 ab 的速度。可能统一改成 \Delta S_k 更好，符号和 Section 2 统一。
Since the trapezoidal formula has only second-order accuracy, we require a higher-precision method to compute $S_{target}$. Here, we employ the Gauss-Lobatto and Gauss-Legendre quadrature formulas, both of which have third-order algebraic accuracy. From these, we obtain corrections at the edge endpoints $a$, and we summarize these corrections---expressed in terms of position and velocity---in Table.\ref{table_guass}.

% 本质上通过增加节点上的自由度，利用节点多项式来描述界面的信息，避免高精度自由的变多.
 % we put forward a third-order rule for selecting quadrature points and assigning weights for numerical integration under a moving mesh framework.
\begin{table}[htbp]
\centering\label{t11}
\small
\renewcommand{\arraystretch}{1.2}
\setlength{\tabcolsep}{5pt}
\caption{Comparison of Gauss-type quadrature rules with algebraic precision 3 for computing $v_a'$.}\label{table_guass}
\begin{tabular}{@{} l p{6.5cm} c @{}}
\toprule
\textbf{Method} & \textbf{Node positions on $[a,b]$} & \textbf{$v_a'$} \\
\midrule
Three-point Gauss--Lobatto &
$x_1=a,\; x_2=\frac{a+b}{2},\; x_3=b$ &
$ \frac{2}{3}v(a)+\frac{2}{3}v(\frac{a+b}{2})-\frac{1}{3}v(b)$ \\
\addlinespace
Two-point Gauss--Legendre &
$a,\; t_1=\frac{a+b}{2}-\frac{b-a}{2\sqrt{3}},\; t_2=\frac{a+b}{2}+\frac{b-a}{2\sqrt{3}},\; b$ &
$ \frac{1}{2}v(a)+\frac{1}{2}v(t_1)+\frac{1}{2}v(t_2)-\frac{1}{2}v(b)$ \\
\bottomrule
\end{tabular}
\end{table}

\subsection{Asymptotic properties of the nodal velocity correction}
Building upon the edge-based velocity corrections $\Delta S$, we now construct the corrected nodal velocity $\mathbf{u}_q^*$. For a node $q$, each adjacent edge $k=qq'$ gives a desired normal component of the corrected nodal velocity:
\begin{equation}
\mathbf{u}_q^*\cdot\mathbf{n}_k = u_{q,k} + \Delta S_{k,q}.
\end{equation}
where $u_{q,k}=\mathbf{u}_q\cdot\mathbf{n}_k$ is the exact normal velocity of edge $k$ at node $q$  and $\Delta S_{k,q}$ is the area correction term from {\bf Lemma \ref{lemma_1}}. In general, a node $q$ has more than two adjacent edges whose outward normals ${\mathbf{n}_k}$ are not collinear, so the system of constraints is overdetermined and cannot be satisfied simultaneously. Following the least-squares framework of Eq. (\ref{CAVEAT}), we seek $\mathbf{u}_q^*$ that minimizes the weighted residual:
% Since the normals of different edges are not collinear, these requirements cannot be satisfied simultaneously in general.
% Recall Eqs.(\ref{CAVEAT}), then let each edge $S^*= u_q + \Delta S_{k,q}$, ones obtain 

\begin{eqnarray} \label{CAVEAT_modify}
\min_{\mathbf{u}_q^*} \sum_{q'\in\mathcal{N}(q)} l_k \left( \mathbf{u}_q^*\cdot\mathbf{n}_k - (u_{q,k} + \Delta S_{k,q}) \right)^2,\quad k=qq'.
\end{eqnarray} 
Setting $\delta\mathbf{u}_q=\mathbf{u}_q^*- \mathbf{u}_q$ and applying the equations (\ref{CAVEAT_node}), the minimization reduces to a $2 \times 2$ linear system:
%By the Eqs.(\ref{CAVEAT_node}), an distilled formula reads 
\begin{eqnarray} \label{CAVEAT_modify1}
\mathcal{M}_q\,\delta\mathbf{u}_q = \mathbf{b}_q, 
\end{eqnarray} 
where the matrix $\mathcal{M}_q = \sum_{q'\in\mathcal{N}(q)} l_k \mathbf{n}_k \mathbf{n}_k^\top$ is the same as in the original CAVEAT algorithm, and the right-hand side $\mathbf{b}_q = \sum_k l_k \Delta S_{k,q} \mathbf{n}_k$ encodes the area correction.

Let us now investigate the asymptotic properties of $\delta\mathbf{u}_q$ under the assumption of a uniform mesh as follows.
\begin{assumption}\label{assump:mesh}
Let $\{\mathcal{G}_h\}_{h>0}$ be a family of uniform quadrilateral meshes with mesh size $h$. There exist constants $C_l, C_\psi, C_\nabla > 0$, independent of $h$, such that for every interior edge $k$ the following bounds hold:
\begin{equation}\label{eq:mesh_prop}
    |l_k - h| \le C_l\,h^3, \qquad
    |\psi_k| \le C_\psi, \qquad
    |\psi_{k'} - \psi_k| \le C_\nabla\,h \quad
    \text{for any edge $k'$ adjacent to $k$},
\end{equation}
where $l_k$ is the edge length, $\Delta S_{k,q} = h^2\,\psi_k$ is the area correction from Lemma~\ref{lemma_1}, and the last inequality encodes the discrete Lipschitz regularity of $\psi$.
\end{assumption}

\begin{theorem}[Asymptotic properties of $\delta\mathbf{u}_q$] \label{thm:asymptotic} 
Under Assumption~\ref{assump:mesh}, suppose the exact velocity field $\mathbf{u}\in [C^4(\Omega)]^2$. Then there exist constants $C > 0$ and
$h_0 > 0$, independent of the node index, such that for all $h \le h_0$ the following estimates of $\delta\mathbf{u}_q$ hold uniformly for all interior nodes $q$ of $\mathcal{G}_h$:
% Let $\{\mathcal{G}_h\}_{h>0}$ be a family of shape-regular quadrilateral meshes of a bounded domain $\Omega \subset \mathbb{R}^2$. Suppose the exact velocity field $\mathbf{u} \in [C(\Omega)]^2$. {......} Consider the nodal velocity correction $\delta\mathbf{u}_q \in \mathbb{R}^2$ as the solution of the linear system:
% \begin{equation}
% \mathcal{M}_q \delta\mathbf{u}_q = \mathbf{b}_q, \qquad \mathcal{M}_q = \sum_{q'\in\mathcal{N}(q)} l_k \mathbf{n}_k \mathbf{n}_k^\top, \quad \mathbf{b}_q =
%   \sum_{q'\in\mathcal{N}(q)} l_k \Delta S_{k,q} \mathbf{n}_k, \quad k = qq'.
% \end{equation}
% Then there exists a constant $C > 0$, independent of $h$ and $q$, such that the following estimates hold uniformly for all interior nodes $q$ and edges $k$ of ${\mathcal{G}_h}$:
\begin{enumerate}
\item (Magnitude of the correction) \begin{equation}\|\delta\mathbf{u}_q\|\leq C h^2. \end{equation}
\item (Smoothness of the correction) For any two adjacent nodes $q$ and $q'$ sharing an edge $k=qq'$,
\begin{equation}
\|\delta\mathbf{u}_q-\delta\mathbf{u}_{q'} \| \leq C h^3.
\end{equation}
\item (High-order guarantee) For any edge $k = qq'$,
\begin{equation}
    \left| \frac{1}{2}\mathbf{n}_k \cdot \bigl(\delta\mathbf{u}_q + \delta\mathbf{u}_{q'}\bigr)-\Delta S_{k,q} \right| \leq C h^4.
\end{equation}
\end{enumerate}
\end{theorem}

\begin{proof}
\medskip
\noindent\underline{\textbf{Proof of Part 1.}}
Each term in $\mathbf{b}_q$ satisfies
\begin{equation}
    \|l_k\,\Delta S_{k,q}\,\mathbf{n}_k\|
    = l_k\,h^2\,|\psi_k|
    \le (h+C_l\,h^3)\,h^2\,C_\psi
    \le 2\,C_\psi\,h^3,
\end{equation}
where the last step uses $C_l\,h^2\le 1$ for $h\le h_0$ with a sufficiently small $h_0$. Summing over the four edges of $\mathcal{N}(q)$:
\begin{equation}\label{eq:b_bound}
    \|\mathbf{b}_q\| \le 8\,C_\psi\,h^3.
\end{equation}
%Combining with \eqref{eq:Minv_bound}:
Since $\mathcal{M}_q$ is a symmetric positive definite matrix formed by summing four rank-one projections $l_k\mathbf{n}_k\mathbf{n}_k^\top$ with $l_k \sim h$ and the normals spanning $\mathbb{R}^2$, standard spectral analysis shows that its eigenvalues satisfy $c_1 h \le \lambda_{\min}(\mathcal{M}_q) \le \lambda_{\max}(\mathcal{M}_q) \le c_2 h$ for constants $c_1, c_2 > 0$ depending only on the mesh regularity. Thus $\|\mathcal{M}_q^{-1}\| \le c_1^{-1}h^{-1}$. Then
\begin{equation}\label{eq:prop1_est}
    \|\delta\mathbf{u}_q\|
    \le \|\mathcal{M}_q^{-1}\|\,\|\mathbf{b}_q\|
    \le c_1^{-1}\,h^{-1}\cdot 8\,C_\psi\,h^3
    = 8\,c_1^{-1}\,C_\psi\,h^2 := K_1 h^2.
\end{equation}

\medskip
\noindent\underline{\textbf{Proof of Part 2.}}
For any two adjacent nodes $q$ and $q'$, write $\mathcal{M}_{q'} = \mathcal{M}_q+\Delta\mathcal{M}$. Then, subtracting $\mathcal{M}_{q'}\,\delta\mathbf{u}_{q'} = \mathbf{b}_{q'}$ and $\mathcal{M}_q\,\delta\mathbf{u}_q = \mathbf{b}_q$:
\begin{equation}\label{eq:diff_system}
    \mathcal{M}_q\bigl(\delta\mathbf{u}_{q'}-\delta\mathbf{u}_q\bigr)
    = \underbrace{-\Delta\mathcal{M}\,\delta\mathbf{u}_{q'}}_{\text{Term\;I}}
    \;+\;\underbrace{(\mathbf{b}_{q'}-\mathbf{b}_q)}_{\text{Term\;II}}.
\end{equation}
\emph{\textbf{Estimate of Term~I.}}
%The nodes $q$ and $q'$ shares the edge $k=qq'$, and the remaining adjacent edges of each node are distinct. 
On a uniform mesh, for each edge appearing in both neighborhoods the variation of $l_k$ and $\mathbf{n}_k$ over one mesh spacing satisfies $|l_k^{(q')}-l_k^{(q)}|\le C_l'\,h^3$ and $\|\mathbf{n}_k^{(q')}-\mathbf{n}_k^{(q)}\|\le C_n'\,h^2$,
where $C_l'$ and $C_n'$ depend only on the reference geometry and the
smoothness of $\mathbf{u}$. Each such edge contributes at most $(C_l'+C_n')\,h^3$ to $\|\Delta\mathcal{M}\|$. The edges unique to each neighborhood yield contributions of the same magnitude. With at most eight edges involved:
\begin{equation}\label{eq:DM_bound}
    \|\Delta\mathcal{M}\| \le 8(C_l'+C_n')\,h^3 =: C_{\Delta M}\,h^3.
\end{equation}
By \textbf{Part 1}:
\begin{equation}\label{eq:term1}
    \|\Delta\mathcal{M}\,\delta\mathbf{u}_p\|
    \le C_{\Delta M}\,h^3\cdot K_1\,h^2
    = K_1\,C_{\Delta M}\,h^5.
\end{equation}
\emph{\textbf{Estimate of Term~II}.}
On a uniform mesh, for each edge $k$ incident to $q$ there exists a unique
edge $k'$ incident to $q'$ with nearly parallel normal. Since $q$ and $q'$ are separated by distance $h$, the smoothness of the mesh mapping and the discrete Lipschitz condition \eqref{eq:mesh_prop} give
\begin{equation}
    \|\mathbf{n}_{k'}-\mathbf{n}_k\| \le C_n\,h^2, \qquad
    |l_{k'}-l_k| \le C_l\,h^3, \qquad
    |\Delta S_{k',q'}-\Delta S_{k,q}| \le C_\nabla\,h^3.
\end{equation}
Writing $\mathbf{b}_{q'}-\mathbf{b}_q$ as a sum over these four pairs and expanding each difference:
\begin{equation}
    l_{k'}\,\Delta S_{k',q'}\,\mathbf{n}_{k'}
    - l_k\,\Delta S_{k,q}\,\mathbf{n}_k
    = \bigl(l_{k'}\,\Delta S_{k',q'}-l_k\,\Delta S_{k,q}\bigr)\mathbf{n}_{k'}
    + l_k\,\Delta S_{k,q}\,(\mathbf{n}_{k'}-\mathbf{n}_k).
\end{equation}
The first factor satisfies
  $|l_{k'}\Delta S_{k',q'}-l_k\Delta S_{k,q}|
  \le 2h\cdot C_\nabla h^3 + C_l h^3\cdot C_\psi h^2
  \le C_1\,h^4$;
the second term is bounded by $2h\cdot C_\psi h^2\cdot C_n h^2 := C_2\,h^5$.
Each pair thus contributes at most $(C_1+C_2)\,h^4$, and summing over four
pairs:
\begin{equation}\label{eq:term2}
    \|\mathbf{b}_{q'}-\mathbf{b}_q\| \le 4(C_1+C_2)\,h^4 =: C_b\,h^4.
\end{equation}
With a sufficiently small $h_0$ such that  $K_1\,C_{\Delta M}\,h \le 1$, by \textbf{Part 1}, ~\eqref{eq:term1} and ~\eqref{eq:term2}
\begin{equation}
\|\delta\mathbf{u}_{q'}-\delta\mathbf{u}_{q} \| \leq \frac{C_b+1}{c_1} h^3 := K_2 h^3.
\end{equation}
\medskip
\noindent\underline{\textbf{Proof of \textbf{Part 3}.}}
For any edge $k=qq'$, dotting \eqref{CAVEAT_modify1} at node $q$ with $\mathbf{n}_k$ and isolating edge $k$ (where $\mathbf{n}_k\cdot\mathbf{n}_{k}=1$):
\begin{equation}
    \mathbf{n}_k^\top\mathbf{b}_q
    = l_{k}\,(\mathbf{n}_k\cdot\delta\mathbf{u}_q) + R_q,
    \qquad
    R_q := \sum_{\substack{m\in\mathcal{N}(q),\ m\neq q'}}
      l_{qm}(\mathbf{n}_k\cdot\mathbf{n}_{qm})(\mathbf{n}_{qm}\cdot\delta\mathbf{u}_q).
\end{equation}
Adding the analogous identity at $q'$ gives
\begin{equation}\label{eq:sum_iii}
  \mathbf{n}_k^\top(\mathbf{b}_q+\mathbf{b}_{q'})
    = l_{k}\,\mathbf{n}_k\cdot(\delta\mathbf{u}_q+\delta\mathbf{u}_{q'})
    + (R_q+R_{q'}).
\end{equation}
\begin{itemize}
\item \emph{\textbf{Left-hand side:}}\; 
\end{itemize}
\begin{equation}
    \mathbf{n}_k^\top(\mathbf{b}_q+\mathbf{b}_{q'}) = \sum_{m\in \mathcal{N}(q)} l_{qm} \Delta S_{qm,q} \mathbf{n}_k^\top\mathbf{n}_{qm} + \sum_{m'\in \mathcal{N}(q')} l_{q'm'} \Delta S_{q'm',q'} \mathbf{n}_k^\top\mathbf{n}_{q'm'}
\end{equation}
Edge $k$ contributes $2\,l_{k}\,\Delta S_{k,q}$ to $\mathbf{n}_{k}^\top(\mathbf{b}_q+\mathbf{b}_{q'})$. For the remaining edges, pair each $m\in\mathcal{N}(q)\setminus\{q'\}$ with its counterpart $m'\in\mathcal{N}(q')\setminus\{q\}$ having nearly anti-parallel normal projection: $|(\mathbf{n}_{qm}\cdot\mathbf{n}_k)+(\mathbf{n}_{q'm'}\cdot\mathbf{n}_k)| \le C_n\,h^2$. By the same smoothness estimates as in~\textbf{Part 2}, $|l_{q'm'}\Delta S_{q'm',q'}-l_{qm}\Delta S_{qm,q}|\le C_1\,h^4$. Each pair then contributes at most $C_\psi h^2\cdot 2h\cdot C_n h^2 + C_1 h^4 \le C_3\,h^5$. Summing over three pairs:
\begin{equation}\label{eq:lhs_iii}
    \bigl|\mathbf{n}_k^\top(\mathbf{b}_q+\mathbf{b}_{q'})
    - 2\,l_{k}\,\Delta S_{k,q}\bigr| \le 3\,C_3\,h^5.
\end{equation}
\begin{itemize}
\item \emph{\textbf{Residual $R_q+R_{q'}$:}}\; 
\end{itemize}
Write $\delta\mathbf{u}_{q'}=\delta\mathbf{u}_q+\mathbf{e}$ with
$\|\mathbf{e}\|\le K_2\,h^3$ by \textbf{Part 2}. Then 
\begin{align}
    R_q + R_{q'}
    &= \sum_{\substack{m\in\mathcal{N}(q),\ m\neq q'}}
        l_{qm}\,(\mathbf{n}_k\cdot\mathbf{n}_{qm})\,
        (\mathbf{n}_{qm}\cdot\delta\mathbf{u}_q)
      +\sum_{\substack{m'\in\mathcal{N}(q'),\ m'\neq q}}
        l_{q'm'}\,(\mathbf{n}_k\cdot\mathbf{n}_{q'm'})\,
        (\mathbf{n}_{q'm'}\cdot\delta\mathbf{u}_{q'})
    \notag\\[4pt]
    &= \underbrace{
        \sum_{\substack{m\in\mathcal{N}(q),\ m\neq q'}}
        l_{qm}\,(\mathbf{n}_k\cdot\mathbf{n}_{qm})\,
        (\mathbf{n}_{qm}\cdot\delta\mathbf{u}_q)
      +\sum_{\substack{m'\in\mathcal{N}(q'),\ m'\neq q}}
        l_{q'm'}\,(\mathbf{n}_k\cdot\mathbf{n}_{q'm'})\,
        (\mathbf{n}_{q'm'}\cdot\delta\mathbf{u}_{q})
      }_{=:\,\Sigma}
    \notag\\[4pt]
    &\quad+\;\underbrace{
        \sum_{\substack{m'\in\mathcal{N}(q'),\ m'\neq q}}
          l_{q'm'}\,(\mathbf{n}_k\cdot\mathbf{n}_{q'm'})\,
          (\mathbf{n}_{q'm'}\cdot\mathbf{e})
      }_{=:\,E'}.
    \label{eq:Rsum_split}
\end{align}
where $\Sigma$ collects terms in $\delta\mathbf{u}_q$ and $|E'|\le 3\cdot 2h\cdot K_2 h^3 = 6K_2\,h^4$. The sum $\Sigma$ pairs in the same way: using $|(\mathbf{n}_{qm}\cdot\mathbf{n}_k)+(\mathbf{n}_{q'm'}\cdot\mathbf{n}_k)| \le C_n\,h^2$,
$|l_{q'm'}-l_{qm}|\le C_l h^3$, and $\|\mathbf{n}_{q'm'}-\mathbf{n}_{qm}\|\le C_n h^2$, each pair contributes at most $C_4\,K_1\,h^5$. Hence
$|\Sigma|\le 3\,C_4\,K_1\,h^5$. The term $E'$ admits the same pairing with $K_2 h^3$ replacing $K_1 h^2$, giving $|E'|\le 3\,C_4\,K_2\,h^6$. Altogether,
\begin{equation}\label{eq:Rsum_iii}
    |R_q+R_{q'}| \le C_5\,h^5.
\end{equation}
\begin{itemize}
\item \emph{\textbf{Combining:}}\; 
\end{itemize}
  Substituting \eqref{eq:lhs_iii} and \eqref{eq:Rsum_iii} into
  \eqref{eq:sum_iii} and dividing by $l_{k}\ge h/2$:
\begin{equation}
   \bigl|\mathbf{n}_k\cdot(\delta\mathbf{u}_q+\delta\mathbf{u}_{q'})
    - 2\,\Delta S_{k,q}\bigr|
    \le \frac{3C_3\,h^5 + C_5\,h^5}{h/2}
    = 2(3C_3+C_5)\,h^4 =: K_3\,h^4.
\end{equation}

Then, let $C = \max (K_1, K_2, K_3/2)$, the proof is completed.
\end{proof}

\subsection{Algorithm workflow}
% Runge Kutta

Building on the area correction analyzed in Section~\ref{area-correction} and the asymptotic properties established in Theorem~\ref{thm:asymptotic}, we now assemble the complete algorithm workflow of our high-order rectilinear Lagrangian method. Time-stepping is performed through a fourth-order Runge-Kutta scheme.

\begin{algorithm}[H]
  \SetAlgoLined
  \caption{High-order rectilinear Lagrangian scheme based on the GCL}
  \label{alg:main}
  \SetKwFunction{Reconstruct}{ReconstructVelocity}
  \SetKwProg{Fn}{Function}{:}{end}

  \KwIn{initial node positions $\{\mathbf{x}_q^{0}\}$, velocity field $\mathbf{u}(\mathbf{x})$, final time $T$, CFL number}
  \KwOut{node positions $\{\mathbf{x}_q^{N}\}$ at $t=T$}
  \BlankLine

  $t \leftarrow 0$,\quad $\mathbf{x}_q \leftarrow \mathbf{x}_q^{0}$\;
  \While{$t < T$}{
    $\Delta t \leftarrow \min\!\bigl(\mathrm{CFL}\cdot h_{\min}/v_{\max},\;T-t\bigr)$\;
    $\{\mathbf{k}_{1,q}\} \leftarrow \Reconstruct(\{\mathbf{x}_q\},\mathbf{u})$\;
    $\{\mathbf{k}_{2,q}\} \leftarrow \Reconstruct(\{\mathbf{x}_q + \tfrac{\Delta t}{2}\mathbf{k}_{1,q}\},\mathbf{u})$\;
    $\{\mathbf{k}_{3,q}\} \leftarrow \Reconstruct(\{\mathbf{x}_q + \tfrac{\Delta t}{2}\mathbf{k}_{2,q}\},\mathbf{u})$\;
    $\{\mathbf{k}_{4,q}\} \leftarrow \Reconstruct(\{\mathbf{x}_q + \Delta t\,\mathbf{k}_{3,q}\},\mathbf{u})$\;
    $\mathbf{x}_q \leftarrow \mathbf{x}_q + \tfrac{\Delta t}{6}\bigl(\mathbf{k}_{1,q}+2\mathbf{k}_{2,q}+2\mathbf{k}_{3,q}+\mathbf{k}_{4,q}\bigr)$\;
    $t \leftarrow t + \Delta t$\;
  }
  \Return $\{\mathbf{x}_q\}$\;

  \BlankLine
  \Fn{\Reconstruct{$\{\mathbf{x}_q\},\,\mathbf{u}$}}{
    \ForEach{node $q$}{
      $\mathcal{M}_q \leftarrow \mathbf{0}$,\quad $\mathbf{b}_q \leftarrow \mathbf{0}$\;
      \ForEach{neighbor $q' \in \mathcal{N}(q)$}{
        $\mathbf{e}_k \leftarrow \mathbf{x}_{q'} - \mathbf{x}_q$,\quad $l_k \leftarrow \|\mathbf{e}_k\|$\;
        $\mathbf{n}_k \leftarrow l_k^{-1}(-e_{k,y},\,e_{k,x})^\top$,\quad $\mathbf{T}_k \leftarrow \mathbf{n}_k\mathbf{n}_k^\top$\;
        $\mathbf{v}_k \leftarrow \mathcal{Q}_{q,q'}[\mathbf{u}]$\quad\text{(see Eq.~\eqref{Lobatto} or Eq.~\eqref{Legendre})}\;
        $\mathcal{M}_q \leftarrow \mathcal{M}_q + l_k\,\mathbf{T}_k$,\quad $\mathbf{b}_q \leftarrow \mathbf{b}_q + l_k\,\mathbf{T}_k\,\mathbf{v}_k$\;
      }
      Solve $\mathcal{M}_q\,\delta\mathbf{u}_q = \mathbf{b}_q$\;
    }
    \Return $\{\delta\mathbf{u}_q\}$\;
  }
\end{algorithm}

\begin{remark}[Choice of edge quadrature]\label{rem:quad}
  The edge-aware velocity $\mathbf{v}_k = \mathcal{Q}_{q,q'}[\mathbf{u}]$ in
  Algorithm~\ref{alg:main} can in fact be defined by any Gauss-type rule. Two natural candidates are listed in
  Table~\ref{table_guass}: the three-point Gauss-Lobatto formula
  \begin{equation}\label{Lobatto}
    \mathcal{Q}_{q,q'}^{\mathrm{GL}}[\mathbf{u}]
    = \tfrac{2}{3}\,\mathbf{u}(\mathbf{x}_q)
    + \tfrac{2}{3}\,\mathbf{u}\!\left(\tfrac{\mathbf{x}_q+\mathbf{x}_{q'}}{2}\right)
    - \tfrac{1}{3}\,\mathbf{u}(\mathbf{x}_{q'}),
  \end{equation}
  and the two-point Gauss-Legendre formula
  \begin{equation}\label{Legendre}
    \mathcal{Q}_{q,q'}^{\mathrm{GL\!eg}}[\mathbf{u}]
    = \tfrac{1}{2}\,\mathbf{u}(\mathbf{x}_q)
    + \tfrac{1}{2}\,\mathbf{u}(\mathbf{t}_1)
    + \tfrac{1}{2}\,\mathbf{u}(\mathbf{t}_2)
    - \tfrac{1}{2}\,\mathbf{u}(\mathbf{x}_{q'}),
  \end{equation}
  with $\mathbf{t}_{1,2} = \tfrac{\mathbf{x}_q+\mathbf{x}_{q'}}{2}
  \mp \tfrac{\mathbf{x}_{q'}-\mathbf{x}_q}{2\sqrt{3}}$.
  A direct Taylor expansion shows that both rules cancel the constant and the
  first-order terms and yield the \emph{identical} leading error
  \begin{equation}\label{eq:edge_quad_err}
    \mathcal{Q}_{q,q'}[\mathbf{u}]
    \;=\; \mathbf{u}(\mathbf{x}_q)
    \;-\; \tfrac{h^2}{12}\,\partial_{\mathbf{e}_k}^{2}\mathbf{u}(\mathbf{x}_q)
    \;-\; \tfrac{h^3}{24}\,\partial_{\mathbf{e}_k}^{3}\mathbf{u}(\mathbf{x}_q)
    \;+\; O(h^4).
  \end{equation}
  This demonstrates that Assumption~\ref{assump:mesh} is well-founded. Consequently, the asymptotic accuracy of the resulting Lagrangian scheme is
  \emph{independent of which of these Gauss-type rules is adopted}, and
  Theorem~\ref{thm:asymptotic} applies to either choice.
  %We use the Gauss--Lobatto form throughout this paper because the edge midpoint $(\mathbf{p}_q+\mathbf{p}_{q'})/2$ is a natural geometric quantity in the rectilinear setting---it coincides with the centroid of the edge-sweep volumes that appear in the discrete GCL identity, so no additional interpolation is needed when the mesh deforms. The Gauss--Legendre nodes $\mathbf{t}_{1,2}$, in contrast, are interior to the edge and require evaluating $\mathbf{v}$ at non-geometric locations that move with the mesh.
\end{remark}
% Delta S 只需要二阶精度(假设 1)，第一种用的积分点更少。

\section{Numerical experiments}\label{secr}
Since any discrete scheme or reconstruction method impacts the accuracy of the numerical scheme, we temporarily avoid using them to obtain the contact velocities at each point along the edges of the cells, opting instead to compute these velocities directly from the exact solution. This also allows us to demonstrate the feasibility of our method more quickly. This section presents two test problems that together validate the proposed high-order rectilinear Lagrangian scheme. Additionally, we compare the two Gauss-type edge quadratures introduced in Remark~\ref{rem:quad}.

\subsection{Smooth isentropic vortex}
\label{sec:vortex_test}
To assess the geometric conservation property of the proposed scheme, we consider a smooth two-dimensional vortex defined on $\Omega = [-10,10]^2$ by the analytic velocity field
\begin{equation}\label{eq:vortex_velocity}
    \mathbf{u}(x,y)
    \;=\; \frac{5}{2\pi}\,\bigl(-y,\;x\bigr)\,
          \exp\!\left(\frac{1 - x^2 - y^2}{2}\right).
\end{equation}
A direct calculation gives $\nabla\cdot\mathbf{u} \equiv 0$, hence the flow is incompressible and every Lagrangian cell preserves its area along trajectories. Setting the initial density to $\rho_0\equiv 1$, the exact solution is
\begin{equation}\label{eq:exact_density}
    \rho(x, y,t) \;\equiv\; 1, \qquad t \in [0, T].
\end{equation}
Then, with the numerical density $\rho_c^n = m_c / |\Omega_{c}^n|$ on cell $c$ at time $t^n$, we calculate the $L_\infty$ and $L_2$ Errors as  
\begin{equation}\label{err}
  \mathrm{Err}(L_\infty) := \max_c |\rho_c^n - 1|,\quad \mathrm{Err}(L_2):=\sqrt{\sum_{c}(\rho_c^n - 1)^{2}|\Omega_{c}^n|},
\end{equation}
and report the results at $t=0.1$ in Table~\ref{tab:convergence}. Both Gauss-type quadratures achieve fourth-order convergence, confirming the theoretical equivalence established in Remark~\ref{rem:quad}.
\begin{table}[!ht]
    \centering
    \caption{Numerical convergence results of smooth isentropic vortex at $t= 0.1$.}
    \label{tab:convergence}
    \vspace{0.5em}
    \begin{tabular}{l ccccl cccc}
        \toprule
        & \multicolumn{4}{c}{Gauss-Lobatto} & & \multicolumn{4}{c}{Gauss-Legendre} \\
        \cmidrule(lr){2-5} \cmidrule(lr){7-10}
        $h$ & $L_\infty$ & Order & $L_2$ & Order & & $L_\infty$ & Order & $L_2$ & Order\\
        \midrule
        0.4 &2.0771E-04  & ---  &2.6943E-05 & --- & &2.1284E-04  & --- &2.7555E-05 & ---  \\
        0.2 &1.5225E-05  &3.77 &1.8358E-06 &3.88 & &1.5651E-05  &3.77 &1.8830E-06 &3.87  \\
        0.1 &9.8058E-07  &3.96 &1.1726E-07 &3.97 & &1.0089E-06  &3.96 &1.2036E-07 &3.97 \\
        0.05 &6.2154E-08  &3.98 &7.3689E-09 &3.99 & &6.3964E-08  &3.98 &7.5653E-09 &3.99  \\
        \bottomrule
    \end{tabular}
\end{table}

\begin{figure}[!ht]
\centering  %图片全局居中
\begin{subfigure}
\centering
\includegraphics[width=6.5cm]{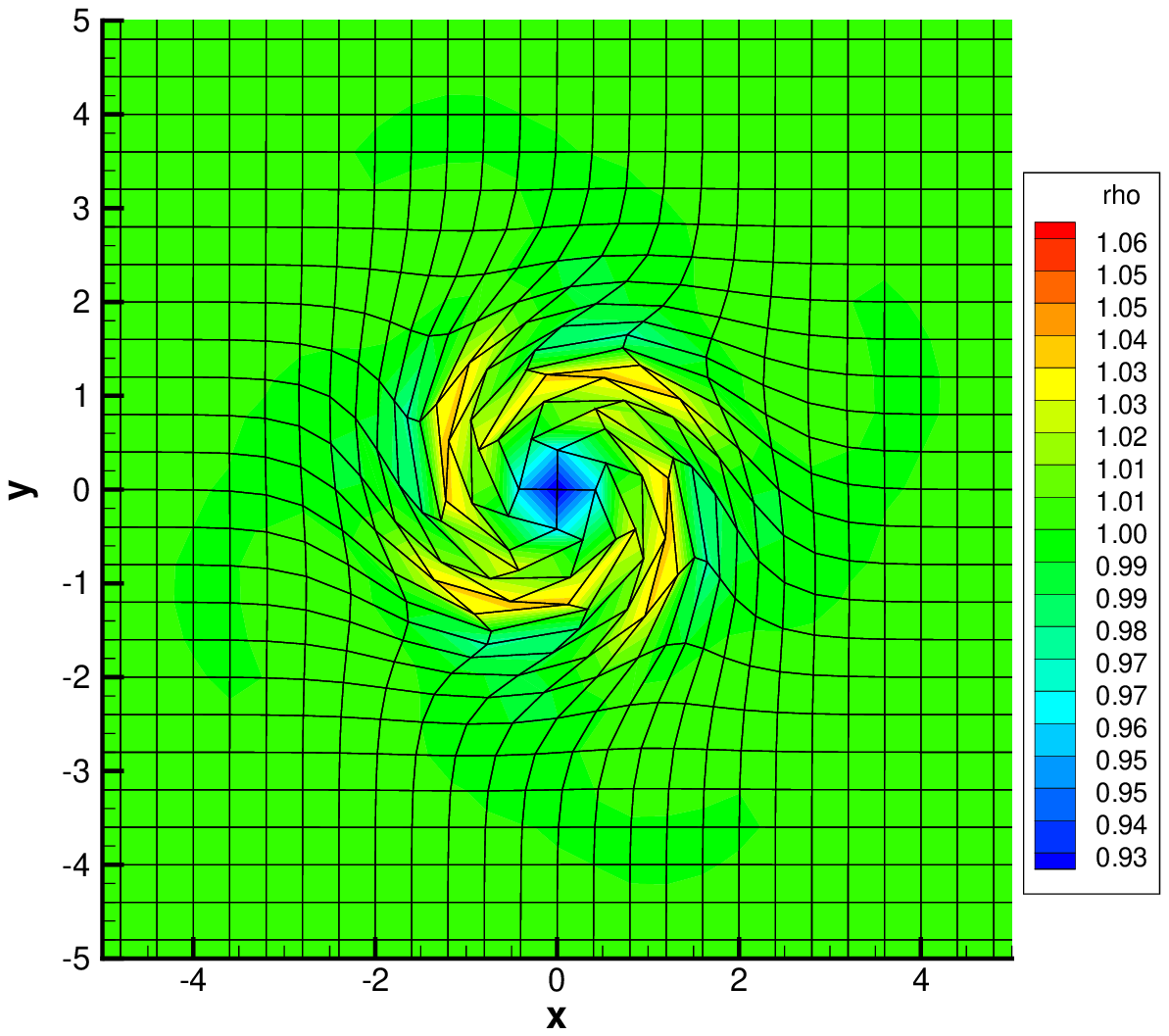}
\end{subfigure}
\qquad \qquad
\begin{subfigure}
\centering
\includegraphics[width=6.5cm]{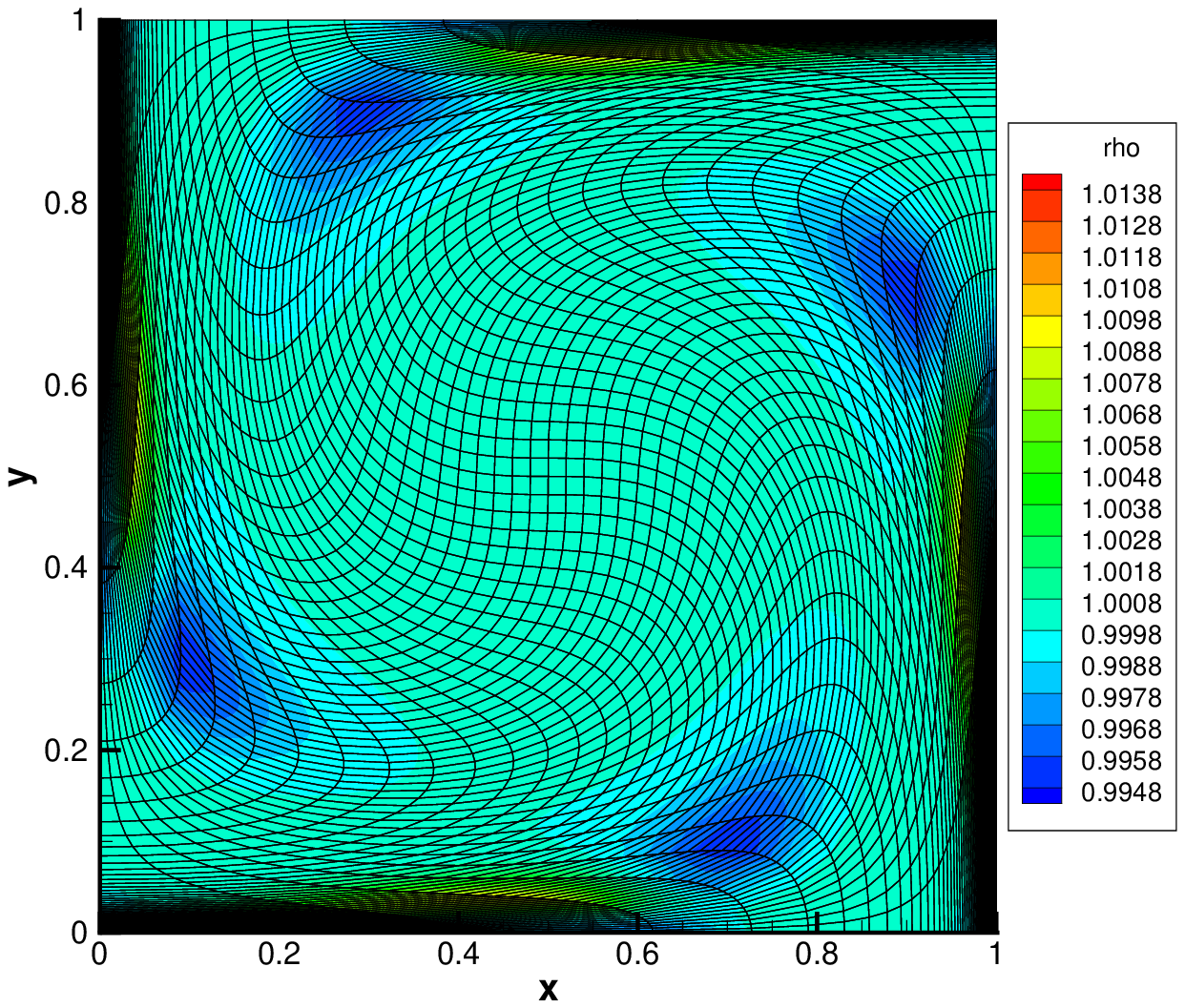}
\end{subfigure}
\caption{Two test cases with 2, 500 quadrilateral meshes. Left: the smooth isentropic vortex ($[-10, 10]^2$, $t=2.5$); Right: the Taylor-Green vortex ($[0, 1]^2$, $t=1.0$).}
\label{Fig.T1}
\end{figure}

\subsection{Taylor-Green vortex}
We then cosider the Taylor-Green vortex flow, which represents a steady-state solution of the incompressible Navier-Stokes equations \cite{Boscheri}. The computational domain is $\Omega = [0, 1]^2$ with periodic boundary conditions imposed on every side. The analytic velocity field is given by
\begin{equation}\label{eq:tg_velocity}
    \mathbf{u}(x,y)
    \;=\; (\sin(\pi x)\cos(\pi y), -\cos(\pi x)\sin(\pi y)).
\end{equation}
Setting the initial density to $\rho_0\equiv 1$. As the flow is divergence-free, the exact density remains constant:
\begin{equation}\label{eq:tg_exact}
      \rho(x, y, t) \;\equiv\; 1, \qquad t \in [0, T].
\end{equation}
Unlike the previous test, the velocity field \eqref{eq:tg_velocity} does not decay at the boundary. This configuration tests the scheme's ability to preserve geometric conservation under periodic boundary conditions. We compute the $L_\infty$ and $L_2$ errors using \eqref{err} and report the results at $t=0.1$ in Table~\ref{tab:tg_convergence}. Both Gauss-type quadratures again achieve fourth-order convergence, demonstrating that the scheme maintains its high-order accuracy under periodic boundary conditions.

\begin{table}[!ht]
    \centering
    \caption{Numerical convergence results of Taylor-Green vortex at $t= 0.1$.}
    \label{tab:tg_convergence}
    \vspace{0.5em}
    \begin{tabular}{l ccccl cccc}
        \toprule
        & \multicolumn{4}{c}{Gauss-Lobatto} & & \multicolumn{4}{c}{Gauss-Legendre} \\
        \cmidrule(lr){2-5} \cmidrule(lr){7-10}
        $h$ & $L_\infty$ & Order & $L_2$ & Order & & $L_\infty$ & Order & $L_2$ & Order\\
        \midrule
        0.04 &1.3362E-06  & ---  &4.5332E-07 & --- & &1.2860E-06  & --- &4.5243E-07 & ---  \\
        0.02 &8.3831E-08  &3.99 &2.8585E-08 &3.99 & &8.0714E-08  &3.99 &2.8539E-08 &3.99  \\
        0.01 &5.2768E-09  &3.99 &1.7905E-09 &4.00 & &5.0829E-09  &3.99 &1.7878E-09 &4.00 \\
        0.005 &3.3307E-10  &3.99 &1.1198E-10 &4.00 & &3.1974E-10  &3.99 &1.1181E-10 &4.00  \\
        \bottomrule
    \end{tabular}
\end{table}

\section{Conclusions}
We propose an effective mesh moving strategy that enables high-order Lagrangian methods on quadrilateral meshes. The method is supported by rigorous theoretical analysis and validated through numerical experiments. We believe this approach will advance research in computational fluid dynamics algorithms.
%The feasibility of the method stems from the perfect combination of solid theoretical foundations and numerical results. We believe that this method will bring new dimensions to the research of computational fluid dynamics algorithms in the future. 

In fact, we have been highly motivated to find a high-order discrete numerical scheme that can serve as a vehicle for our method. Meanwhile, we will integrate this method with the traditional Riemann solver to overcome complicated challenges like shock waves. Also, it should be pointed out that our goal is not to develop a high-order mesh moving method that outperforms others in every aspect, but rather to provide a reliable and theoretically grounded option for numerical simulations within the Lagrangian framework. %but to provide a trustworthy option for numerical simulations within the Lagrangian framework. 

\section*{Declaration of competing interest}
The authors declare that they have no known competing financial interests or personal relationships that could have
appeared to influence the work reported in this paper.

\section*{Data availability}
No data was used for the research described in the article.

\section*{Acknowledgements}
{
% We thank anonymous referees for their insightful comments,
% constructive feedback, and valuable suggestions that have significantly improved the quality and clarity of this manuscript. 

This project was supported by National Natural Science Foundation of China (Grant No. 12501573).
% and China Postdoctoral Science Foundation (Grant No. 2024M760059).
}

% \section*{References}

 \end{document}